\documentclass[11pt]{amsart}
\usepackage{amssymb}
\usepackage[all]{xy}
\usepackage{mathrsfs}
\usepackage{amsmath,amssymb,amsthm}
\usepackage{extarrows}
\newtheorem{theorem}{Theorem}[section]
\newtheorem{lemma}[theorem]{Lemma}

\newtheorem{definition}[theorem]{Definition}

\newtheorem{remark}[theorem]{Remark}

\def\to{\rightarrow}

\def\c{\mathcal}

\def\r{\mathrm}
\def\bb{\mathbb}

\begin{document}
\title[A modification of Hardy-Littlewood maximal-function]{A modification of Hardy-Littlewood maximal-function on Lie groups}
\author[M.M. Sadr]{Maysam Maysami Sadr}
\address{Institute for Advanced Studies in Basic Sciences, Zanjan, Iran}
\email{sadr@iasbs.ac.ir}
\begin{abstract}
For a real-valued function $f$ on a metric measure space $(X,d,\mu)$ the Hardy-Littlewood maximal-function of $f$ is given by the following 
`supremum-norm': $$Mf(x):=\sup_{r>0}\frac{1}{\mu(\c{B}_{x,r})}\int_{\c{B}_{x,r}}|f|d\mu.$$
In this note, we replace the supremum-norm on parameters $r$ by $\c{L}_p$-norm with weight $w$ on parameters $r$
and define Hardy-Littlewood integral-function $I_{p,w}f$. It is shown that $I_{p,w}f$ converges pointwise
to $Mf$ as $p\to\infty$. Boundedness of the sublinear operator $I_{p,w}$ and continuity of the function $I_{p,w}f$ 
in case that $X$ is a Lie group, $d$ is a left-invariant metric, and $\mu$ is a left Haar-measure (resp. right Haar-measure) are studied.
\end{abstract}
\maketitle
\section{Introduction}\label{2212290825}
Maximal sublinear operators and their boundedness properties are one of the main tools in various aspects of Fourier Analysis
on Euclidean spaces $\bb{R}^n$ \cite{Stein1}. The prototype of these operators is Hardy-Littlewood maximal-function $M$ defined by
\begin{equation}\label{2301140934}Mf(x):=\sup_{r>0}\frac{1}{2r}\int_{x-r}^{x+r}|f(y)|dy\end{equation}
for any locally integrable function $f$ on $\bb{R}$. On other hand, there has been many attempts to extend various classical
results of Fourier Analysis for general metric measure spaces and in particular for Riemannian manifolds and Lie groups.
We only mention a few recent works with different flavors: \cite{Chousionis1,Forni1,Hormozi1,Ionescu1,Papageorgiou1,Santagati1,YangLi1}.
One of the problems concerning such extensions, is to define appropriate maximal
operators with good boundedness properties. In \cite{Sadr1} we considered an abstract and unified approach to
$(1,1)$-weak type boundedness of Hardy-Littlewood maximal-function operators.
The main idea of the present note is to replace `supremum' in the definitions of maximal operators by appropriate integrals on
parameter-spaces, in order to find some affable substitutes for maximal operators.
In this note we apply this idea to Hardy-Littlewood maximal-function operator on metric measure spaces. For instance,
our modified version of (\ref{2301140934}) becomes
$$I_{p,w}f(x):=\Bigg(\int_0^\infty\Big(\frac{1}{2r}\int_{x-r}^{x+r}|f(y)|dy\Big)^pw(r)dr\Bigg)^\frac{1}{p},$$
where $w$ is an integrable function of $r$ and $1\leq p<\infty$. We call $I_{p,w}f$ the Hardy-Littlewood integral-function.
In $\S$\ref{2212290826}, we give the definition of integral-function operators $I_{p,w}$ and prove that $\lim_{p\to\infty}I_{p,w}f(x)=Mf(x)$.
In $\S$\ref{2212290827} (resp. $\S$\ref{2212290828}), we prove that for any (non-compact) Lie group $G$ with a left-invariant metric
and a left-invariant measure (resp. right-invariant measure) $I_{p,w}$ is $(\c{L}_q(G),\c{L}_q(G))$-bounded for $1\leq p\leq q\leq\infty$
(and suitable $w$). We also show that $I_{p,w}f$ is almost every-where continuous for $f\in\c{L}_p(G)$ and $1\leq p<\infty$.
We hope to give some applications in a future work.

\textbf{Acknowledgement.} The author would like to express his sincere gratitude to Dr. Mahdi Hormozi for
valuable discussions on various aspects of Fourier Analysis.
\section{The main definition}\label{2212290826}
Let $X$ be a metric space with an unbounded distance function denoted by $d$. The open ball with center $x\in X$ and radius $r>0$ is denoted by
$\c{B}_{x,r}$. We have the following easy lemma.
\begin{lemma}\label{2212240755}
Let $\nu$ be a Borel measure on $X$ which is finite on bounded subsets.
Then the function $(x,r)\mapsto\nu(\c{B}_{x,r})$ from $X\times(0,\infty)$ into $[0,\infty)$
is lower semi-continuous and the function $r\mapsto\nu(\c{B}_{x,r})$ is left continuous.
If for every $x,r$ we have $\nu\{y:d(x,y)=r\}=0$ ($\text{e.g.}$ $X$ is a Riemannian manifold and $d,\nu$ are the canonical metric and
measure on $X$) then the function $(x,r)\mapsto\nu(\c{B}_{x,r})$ is continuous.\end{lemma}
\begin{proof}Let $(x_n)_n$ and $(r_n)_n$ be sequences respectively in $X$ and $(0,\infty)$ such that $x_n\to x$ and $r_n\to r>0$.
We have $\cap_n\cup_{k\geq n}(\c{B}_{x,r}\setminus\c{B}_{x_k,r_k})=\emptyset$
and hence $\nu(\c{B}_{x,r}\setminus\c{B}_{x_n,r_n})\to0$. Thus if $\epsilon>0$ then for sufficiently
large $n$ we have $\nu(\c{B}_{x,r})-\epsilon<\nu(\c{B}_{x,r}\cap\c{B}_{x_n,r_n})\leq\nu(\c{B}_{x_n,r_n})$.
This shows the desired lower semi-continuity. Since $r\mapsto\nu(\c{B}_{x,r})$ is an increasing function, the lower semi-continuity
implies the desired left continuity. We have $$\cap_n\cup_{k\geq n}(\c{B}_{x_k,r_k}\setminus\c{B}_{x,r})\subseteq\{y:d(x,y)=r\}.$$
Thus if $\nu\{y:d(x,y)=r\}=0$ then $\nu(\c{B}_{x_n,r_n}\setminus\c{B}_{x,r})\to0$ and hence for sufficiently large $n$
we have $\nu(\c{B}_{x_n,r_n})-\epsilon<\nu(\c{B}_{x_n,r_n}\cap\c{B}_{x,r})\leq\nu(\c{B}_{x,r})$. \end{proof}
Let $\mu$ be a Borel measure on $X$ with $\mu(X)=\infty$ and such that for any nonempty bounded open subset $U$ of $X$, $0<\mu(U)<\infty$.
We denote by $\c{F}_{loc}(X)$ the set of measurable functions $f$ on $X$ such that $\int_U|f|d\mu<\infty$
for every bounded Borel subset $U$. For any $f\in\c{F}_{loc}(X)$ the averaging-function ${A}f$ of $f$ is defined by
$${A}f:X\times(0,\infty)\to[0,\infty),\quad{A}f(x,r):=\frac{1}{\mu(\c{B}_{x,r})}\int_{\c{B}_{x,r}}|f|d\mu.$$
By Lemma \ref{2212240755} the functions $(x,r)\mapsto\int_{\c{B}_{x,r}}|f|d\mu$ and $(x,r)\mapsto\frac{1}{\mu(\c{B}_{x,r})}$ are measurable.
Thus ${A}f$ is measurable. The Hardy-Littlewood maximal-function ${M}f$ of $f$ is a measurable function on $X$ defined by
$${M}f:X\to[0,\infty],\quad{M}f(x):=\sup_{r>0}{A}f(x,r).$$
Thus ${M}f(x)$ is just the supremum-norm of the function $r\mapsto{A}f(x,r)$.
Our main idea is to replace the supremum-norm by an $\c{L}_p$-norm:
\begin{definition}Let $X,d,\mu$ be as above. Let $w$ denote a nonnegative measurable function on $(0,\infty)$ with
$\|w\|:=\int_0^\infty w(r)dr<\infty$ and such that $w$ is also almost everywhere nonzero. We call $w$ a radius-weight.
Denote by $\hat{w}$ the finite measure on $(0,\infty)$ with density $w$. For any $f\in\c{F}_{loc}(X)$ the Hardy-Littlewood integral-function
$I_{p,w}f$ of type $(p,w)$,  $1\leq p\leq\infty$, is defined to be the measurable function on $X$ given by
$$I_{p,w}f:X\to[0,\infty],\quad I_{p,w}f(x):=\|r\mapsto{A}f(x,r)\|_{\c{L}_p(\hat{w})}.$$\end{definition}
More explicitly, for $1\leq p<\infty$ we have $$I_{p,w}f(x):=\Big(\int_0^\infty w(r)(Af(x,r))^pdr\Big)^{1/p}.$$
By Lemma \ref{2212240755}, $r\mapsto Af(x,r)$ is left continuous. Thus for $p=\infty$ we have $$I_{\infty,w}f(x)=Mf(x).$$
Since $A$ is sublinear, $I_{p,w}$ is sublinear on $\c{F}_{loc}(X)$.
\begin{lemma}\label{2212260802}
Let $\theta$ be a finite measure on a measurable space $T$ and let $\phi:T\to[0,\infty)$ be measurable. Then
$\|\phi\|_{\c{L}_\infty(\theta)}=\lim_{p\to\infty}\|\phi\|_{\c{L}_p(\theta)}$.\end{lemma}
\begin{proof} We denote $\|\cdot\|_{\c{L}_p(\theta)}$ by $\|\cdot\|_p$. Suppose $\|\phi\|_\infty<\infty$. Without lost of generality
assume $\|\phi\|_\infty=1$ and $\theta(T)=1$. We have $\limsup_p\|\phi\|_p\leq1$. For $\epsilon>0$ let $S_\epsilon:=\{t:\phi(t)>1-\epsilon\}$.
Then $(1-\epsilon)\theta(S_\epsilon)^\frac{1}{p}\leq\|\phi\|_p$. Since $\theta(S_\epsilon)\neq0$ we have $(1-\epsilon)\leq\liminf_p\|\phi\|_p$,
and hence $1\leq\liminf_p\|\phi\|_p$. Thus $\|\phi\|_\infty=\lim_p\|\phi\|_p$. Now suppose $\|\phi\|_\infty=\infty$. Let $S'_n:=\{t:\phi(t)\leq n\}$.
By the first part of the proof we have $\|\phi|_{S'_n}\|_\infty=\lim_p\|\phi|_{S'_n}\|_p$. Thus $\|\phi|_{S'_n}\|_\infty\leq\liminf_p\|\phi\|_p$.
Since $\sup_n\|\phi|_{S'_n}\|_\infty=\infty$ we have $\lim_p\|\phi\|_p=\infty$.\end{proof}
\begin{theorem}\label{2301080828}For any $f\in\c{F}_{loc}(X)$ and every $x\in X$ we have
$$\lim_{p\to\infty}I_{p,w}f(x)=Mf(x).$$\end{theorem}
\begin{proof}It follows
from Lemma \ref{2212260802}, with $T=(0,\infty)$, $\theta=\hat{w}$, $\phi=Af(x,\cdot)$.\end{proof}
\begin{theorem}\label{2301080828+1/2}For $q\in[1,\infty)$ the following statements are equivalent:
\begin{enumerate}\item[(i)] $M$ is  $(\c{L}_q(X),\c{L}_q(X))$-bounded.
\item[(ii)] The family $\{I_{p,w}\}_{1\leq p<\infty}$ is uniformly $(\c{L}_q(X),\c{L}_q(X))$-bounded.
\item[(iii)] There exists a sequence $(p_n)_n$ in $[1,\infty)$ such that $p_n\to\infty$ and such that the family
$\{I_{p_n,w}\}_{n}$ is uniformly $(\c{L}_q(X),\c{L}_q(X))$-bounded.\end{enumerate}\end{theorem}
\begin{proof}Since $Af(x,r)\leq Mf(x)$ we have $I_{p,w}f(x)\leq\|w\|^{\frac{1}{p}}Mf(x)$. Thus
$\|I_{p,w}f\|_{\c{L}_q(X)}\leq\|w\|^{\frac{1}{p}}\|Mf\|_{\c{L}_q(X)}$. This shows (i)$\Rightarrow$(ii). By Theorem \ref{2301080828} and
Fatou's Lemma we have $\int_X(Mf)^qd\mu\leq\liminf_p\int_X(I_{p,w}f)^qd\mu$. This shows (iii)$\Rightarrow$(i). (ii)$\Rightarrow$(iii)
is trivial.\end{proof}
It is not hard to see that the statement of Theorem \ref{2301080828+1/2} is valid if the term `$(\c{L}_q(X),\c{L}_q(X))$-bounded'
is replaced by `$(\c{L}_q(X),\c{L}_q(X))$-weak-bounded'. In the case that $X=\bb{R}^n$, $d$ the standard Euclidean distance, and $\mu$
the $n$-dimensional Lebesgue-measure, it is well-known that $M$ is $(\c{L}_q(\bb{R}^n),\c{L}_q(\bb{R}^n))$-bounded for $1<q\leq\infty$ and also
$(\c{L}_1(\bb{R}^n),\c{L}_1(\bb{R}^n))$-weak-bounded (\cite{Stein1}). Thus the latter statement is valid with $M$ replaced by $I_{p,w}$.
We will see from Theorem \ref{2212310710} that $I_{1,w}$ is also  $(\c{L}_1(\bb{R}^n),\c{L}_1(\bb{R}^n))$-bounded.
The proof of the next result follows from the definition of $I_{p,w}$, and is omitted.
\begin{theorem} In the case that $X=\bb{R}^n$, for any nonnegative Schwartz test-function $f$ on $\bb{R}^n$ and every $p\in[1,\infty)$,
$I_{p,w}f$ is continuously $[p]$ times differentiable, where $[p]$ denotes the greatest integer $\leq p$.\end{theorem}
We will see from Theorem \ref{2301150935} that for any $p\in[1,\infty)$ and every $f\in\c{L}_p(\bb{R}^n)$, the function $I_{p,w}f$ is
almost every-where continuous.
\begin{remark}It is clear that the above formalism of `replacing supremum-norm by $\c{L}_p$-norm on parameter-space'
may be applied to almost all maximal sublinear operators of any kind. One can also work in an abstract framework as in \cite{Sadr1}.
In this note we only consider the formalism for centered-ball Hardy-Littlewood maximal-function operators.
\end{remark}
\section{$I_{p,w}$ on Lie groups (I)}\label{2212290827}
With the notations $X,d,\mu,w$ as in $\S$\ref{2212290826}, suppose that $X=G$ is a non-compact Lie group and suppose that $d$ and $\mu=\lambda$
denote the distance function and the measure canonically induced by a left-invariant
Riemannian  metric on $G$. Thus $d$ is a left-invariant metric and $\lambda$ is a left Haar-measure on $G$.
The space $\c{F}_{loc}(G)$ coincides with the vector space of locally integrable functions on $G$.
By Lemma \ref{2212240755}, we know that for any $f\in\c{F}_{loc}(G)$ the function $Af:G\times(0,\infty)\to[0,\infty)$ is continuous.
\begin{lemma}\label{2101302224}For any $f\in\c{F}_{loc}(G),r\in(0,\infty),p\in[1,\infty)$ we have
\begin{equation}\label{2212300851}
\|Af(\cdot,r)\|_{\c{L}_p(G)}\leq\Big(\frac{1}{\lambda(\c{B}_{e,r})}\int_{\c{B}_{e,r}}\Delta(y^{-1})d\lambda(y)\Big)^{\frac{1}{p}}\|f\|_{\c{L}_p(G)}.
\end{equation} Also we have $\|Af(\cdot,r)\|_{\c{L}_\infty(G)}\leq\|f\|_{\c{L}_\infty(G)}$.\end{lemma}
Here $\Delta$ denotes the modular function of $G$ (\cite{Folland1}), i.e. the unique mapping $\Delta:G\to(0,\infty)$ satisfying
$\lambda(Bx)=\Delta(x)\lambda(B)$ for every $x\in G$ and every Borel subset $B$ of $G$. Note that $\Delta$ is a continuous group-homomorphism.
Thus, it follows from the relatively-compactness of $\c{B}_{e,r}$ in $G$, that the integral in right-hand side of (\ref{2212300851}) is finite.
For unimodular groups ($\text{e.g.}$ abelian groups) $\Delta\equiv1$. Thus for unimodular $G$, (\ref{2212300851}) becomes
$$\|Af(\cdot,r)\|_{\c{L}_p(G)}\leq\|f\|_{\c{L}_p(G)}.$$
\begin{proof}Suppose that $f\geq0$. For $1\leq p<\infty$, by Jensen's inequality we have\begin{equation*}\begin{split}
\int_{G}(Af(x,r))^pd\lambda(x)&\leq\int_{G}\Big(\frac{1}{\lambda(\c{B}_{x,r})}\int_{\c{B}_{x,r}}f^p(y)d\lambda(y)\Big)d\lambda(x)\\
&=\int_{G}\Big(\frac{1}{\lambda(x\c{B}_{e,r})}\int_{\c{B}_{e,r}}f^p(xy)d\lambda(y)\Big)d\lambda(x)\\
&=\frac{1}{\lambda(\c{B}_{e,r})}\int_{\c{B}_{e,r}}\Big(\int_Gf^p(xy)d\lambda(x)\Big)d\lambda(y)\\
&=\frac{1}{\lambda(\c{B}_{e,r})}\int_{\c{B}_{e,r}}\Big(\Delta(y^{-1})\int_Gf^p(x)d\lambda(x)\Big)d\lambda(y)\\
&=\frac{\|f\|^p_{\c{L}_p(G)}}{\lambda(\c{B}_{e,r})}\int_{\c{B}_{e,r}}\Delta(y^{-1})d\lambda(y).\end{split}\end{equation*}
The case $p=\infty$ is trivial.\end{proof}
\begin{definition}With the above assumptions, the $G$-norm of any radius-weight $w$ is denoted by $\|w\|_G$ and is defined by
$$\|w\|_G:=\int_0^\infty\frac{w(r)}{\lambda(\c{B}_{e,r})}\Big(\int_{\c{B}_{e,r}}\Delta(y^{-1})d\lambda(y)\Big)dr.$$\end{definition}
If $G$ is unimodular then we have $\|w\|_G=\|w\|<\infty$. It is clear that for any $G$ there exist radius-weights
with finite $G$-norm. For instance:
$$w(r)=\Bigg\{\begin{array}{cc}
\frac{e^{-r^2}\lambda(\c{B}_{e,r})}{\int_{\c{B}_{e,r}}\Delta(y^{-1})d\lambda(y)}&\hspace{3mm}\text{if }
1<\frac{1}{\lambda(\c{B}_{e,r})}\int_{\c{B}_{e,r}}\Delta(y^{-1})d\lambda(y) \\
e^{-r^2}&\hspace{3mm}\text{otherwise} \\\end{array}$$
\begin{theorem}\label{2212310710}With assumptions of this section on $G$, suppose that $w$ is a radius-weight with finite $G$-norm.
Then we have $$\|I_{p,w}f\|_{\c{L}_q(G)}\leq\|w\|^{\frac{q-p}{qp}}\|w\|_G^{\frac{1}{q}}\|f\|_{\c{L}_q(G)},
\hspace{5mm}(f\in\c{F}_{loc}(G),1\leq p\leq q\leq\infty).$$\end{theorem}
(Note that, by convention, $\frac{\infty-p}{\infty p}=\frac{1}{p}$ and $\frac{1}{\infty}=0$.)
\begin{proof} By Jensen's Inequality and Lemma \ref{2101302224}, for $q\neq\infty$ we have
\begin{equation*}\begin{split}
\|I_{p,w}f\|^q_{\c{L}_q(G)}&=\int_G\Big(\int_0^\infty w(r)(Af(x,r))^pdr\Big)^{\frac{q}{p}}d\lambda(x)\\
&=\|w\|^{\frac{q}{p}}\int_G\Big(\int_0^\infty\frac{w(r)}{\|w\|}(Af(x,r))^pdr\Big)^{\frac{q}{p}}d\lambda(x)\\
&\leq\|w\|^{\frac{q}{p}}\int_G\int_0^\infty\frac{w(r)}{\|w\|}(Af(x,r))^qdrd\lambda(x)\\
&=\|w\|^{\frac{q-p}{p}}\int_0^\infty w(r)\Big(\int_G(Af(x,r))^qd\lambda(x)\Big)dr\\
&\leq\|w\|^{\frac{q-p}{p}}\int_0^\infty\frac{w(r)\|f\|^q_{\c{L}_q(G)}}{\lambda(\c{B}_{e,r})}\Big(\int_{\c{B}_{e,r}}\Delta(y^{-1})d\lambda(y)\Big)dr\\
&=\|w\|^{\frac{q-p}{p}}\|w\|_G\|f\|^q_{\c{L}_q(G)}.\end{split}\end{equation*}
For $q=\infty$ the desired inequality is easily concluded.\end{proof}
\begin{theorem}\label{2301150935}
With assumptions of this section on $G$, suppose that $w$ is an arbitrary radius-weight.
Let $f\in\c{L}_p(G)$ with $1\leq p<\infty$. Then for any $x\in G$ such that $f$ is essentially bounded on a neighborhood of $x$, $I_{p,w}f$ is continuous at $x$. In particular, $I_{p,w}f$ is continuous almost every-where.\end{theorem}
\begin{proof}Without lost of generality, suppose that $f\geq0$.
Let $\epsilon>0$ be arbitrary and fixed. Choose a positive real number $a$ such that $\int_0^{a}w(r)dr<\epsilon$ and such that
$M:=\r{ess}\sup f|_{\c{B}(x,3a)}<\infty$. Then, for any $y\in G$ with $d(x,y)<a$ we have
\begin{equation}\label{2301080948}\begin{split}
&\int_0^aw(r)\big|Af(x,r)-Af(y,r)\big|^pdr\\
=&\int_0^aw(r)\big|\frac{1}{\lambda(\c{B}_{e,r})}\int_{\c{B}_{e,r}}(f(xz)-f(yz))d\lambda(z)\big|^pdr\\
\leq&\int_0^aw(r)\big(\frac{1}{\lambda(\c{B}_{e,r})}\int_{\c{B}_{e,r}}|f(xz)-f(yz)|d\lambda(z)\big)^pdr\\
\leq&\int_0^aw(r)2^pM^pdr<2^pM^p\epsilon.\end{split}\end{equation}
Choose a positive real number $b$ such that $\frac{1}{\lambda(\c{B}_{e,b})}<\epsilon$. Then, by Jensen's Inequality, for any $y\in G$ we have
\begin{equation}\label{2301080949}\begin{split}
&\int_b^\infty w(r)\big|Af(x,r)-Af(y,r)\big|^pdr\\
\leq&\int_b^\infty w(r)\big(\frac{1}{\lambda(\c{B}_{e,r})}\int_{\c{B}_{e,r}}|f(xz)-f(yz)|d\lambda(z)\big)^pdr\\
\leq&\int_b^\infty\frac{w(r)}{\lambda(\c{B}_{e,r})}\int_{\c{B}_{e,r}}|f(xz)-f(yz)|^pd\lambda(z)dr\\
\leq&\int_b^\infty\frac{w(r)}{\lambda(\c{B}_{e,r})}\int_{G}|f(x\cdot)-f(y\cdot)|^pd\lambda dr\\
\leq&\int_b^\infty\frac{w(r)}{\lambda(\c{B}_{e,r})}2^p\|f\|_{\c{L}_p(G)}^pdr<2^p\|w\|\|f\|_{\c{L}_p(G)}^p\epsilon.\\\end{split}\end{equation}
Since $Af$ is continuous there exists $\delta>0$ such that for any $y\in\c{B}_{x,\delta}$:
$$|Af(x,r)-Af(y,r)|<\epsilon,\quad(a\leq r\leq b),$$ and hence
\begin{equation}\label{2301080950}\int_a^b w(r)\big|Af(x,r)-Af(y,r)\big|^pdr\leq\|w\|\epsilon^p.\end{equation}
If $d(x,y)<\min\{a,\delta\}$ then by (\ref{2301080948}),(\ref{2301080949}),(\ref{2301080950}) we have
\begin{equation*}\begin{split}\big|I_{p,w}f(x)-I_{p,w}f(y)\big|^p&\leq\int_0^\infty w(r)\big|Af(x,r)-Af(y,r)\big|^pdr\\
&\leq\big(2^pM^p\epsilon\big)+\big(\|w\|\epsilon^p\big)+\big(2^p\|w\|\|f\|_{\c{L}_p(G)}^p\epsilon\big).\end{split}\end{equation*}
The proof is complete.\end{proof}
\section{$I_{p,w}$ on Lie Groups (II)}\label{2212290828}
With the notations $X,d,\mu,w$ as in $\S$\ref{2212290826}, suppose that $X=G$ is a non-compact Lie group.
Consider two Riemannian metrics on $G$ such that one of them is left-invariant and another one is right-invariant, and such that
the two metrics coincide on Lie-algebra of $G$. Let $d$ denote the distance function on $G$ induced by the left-invariant metric and
let $\mu=\rho$ denote the measure on $G$ induced by the right-invariant metric. Thus $\rho$ is a right Haar-measure on $G$. If $\lambda$ as in
$\S$\ref{2212290827} denotes the measure induced by the left-invariant metric then we have $\lambda(B)=\rho(B^{-1})$ and
$\rho(xB)=\Delta(x^{-1})\rho(B)$  for every $x\in G$ and Borel subset $B$ of $G$.
\begin{lemma}\label{2301130700}For any $f\in\c{F}_{loc}(G),r\in(0,\infty),p\in[1,\infty]$ we have
\begin{equation*}\|Af(\cdot,r)\|_{\c{L}_p(G)}\leq\|f\|_{\c{L}_p(G)}.\end{equation*}.\end{lemma}
\begin{proof}Suppose that $f\geq0$. For $1\leq p<\infty$, by Jensen's inequality we have\begin{equation*}\begin{split}
\int_{G}(Af(x,r))^pd\rho(x)&\leq\int_{G}\Big(\frac{1}{\rho(\c{B}_{x,r})}\int_{\c{B}_{x,r}}f^p(y)d\rho(y)\Big)d\rho(x)\\
&=\int_{G}\Big(\frac{1}{\Delta(x^{-1})\rho(\c{B}_{e,r})}\int_{x\c{B}_{e,r}}f^p(y)d\rho(y)\Big)d\rho(x)\\
&=\frac{1}{\rho(\c{B}_{e,r})}\int_G\int_{\c{B}_{e,r}}f^p(xy)d\rho(y)d\rho(x)\\
&=\frac{1}{\rho(\c{B}_{e,r})}\int_{\c{B}_{e,r}}\int_Gf^p(xy)d\rho(x)d\rho(y)\\
&=\frac{1}{\rho(\c{B}_{e,r})}\int_{\c{B}_{e,r}}\|f\|^p_{\c{L}_p(G)}d\rho(y)\\
&=\|f\|^p_{\c{L}_p(G)}.\end{split}\end{equation*}The case $p=\infty$ is trivial.\end{proof}
The proof of the following theorem is omitted. It is similar to the proof of Theorem \ref{2212310710} but uses Lemma \ref{2301130700}.
\begin{theorem}With assumptions of this section on $G$, suppose that $w$ is an arbitrary radius-weight.
We have $$\|I_{p,w}f\|_{\c{L}_q(G)}\leq\|w\|^{\frac{1}{p}}\|f\|_{\c{L}_q(G)},
\hspace{5mm}(f\in\c{F}_{loc}(G),1\leq p\leq q\leq\infty).$$\end{theorem}
The statements of Theorem \ref{2301150935} remains valid with the new assumptions of this section on $G$.
The proof is also similar to the proof of Theorem \ref{2301150935}. The only thing that may need an explanation is the relevant modification
of (\ref{2301080949}): We have $\|f\|_{\c{L}_p(G)}^p=\Delta(x^{-1})\int_Gf^p(x\cdot)d\rho$. Thus if we get $y$ so close to $x$ such that
$\Delta(y)\leq2\Delta(x)$ then we have
\begin{equation*}\begin{split}\int_{G}|f(x\cdot)-f(y\cdot)|^pd\lambda dr&=\|f(x\cdot)-f(y\cdot)\|_{\c{L}_p(G)}^p\\
&\leq\Big(\|f(x\cdot)\|_{\c{L}_p(G)}+\|f(y\cdot)\|_{\c{L}_p(G)}\Big)^p\\
&=\Big(\Delta(x)^\frac{1}{p}+\Delta(y)^\frac{1}{p}\Big)^p\|f\|_{\c{L}_p(G)}^p\\
&\leq3^p\Delta(x)\|f\|_{\c{L}_p(G)}^p.\end{split}\end{equation*}
Hence we replace the last line of (\ref{2301080949}) by
$$\leq\int_b^\infty\frac{w(r)}{\lambda(\c{B}_{e,r})}3^p\Delta(x)\|f\|_{\c{L}_p(G)}^pdr<3^p\Delta(x)\|w\|\|f\|_{\c{L}_p(G)}^p\epsilon.$$

\end{document}